\documentclass[12pt,reqno]{amsart}
\usepackage{amsmath,amsfonts,amsthm,bbm,amssymb,color}
\usepackage[usenames,dvipsnames]{xcolor}
\usepackage[T1]{fontenc}
\usepackage{enumerate}
\usepackage{hyperref}
\hypersetup{
     colorlinks   = true,
     citecolor    = blue,
     linkcolor    = blue
}

\usepackage{pdfsync}{\tiny }

\usepackage[left=1in, right=1in, top=1.1in,bottom=1.1in]{geometry}
\numberwithin{equation}{section}

\newcommand{\E}{\mathrm{E}}
\renewcommand{\P}{\mathrm{P}}
\newcommand{\R}{\mathbb{R}}

\newtheorem{thm}{Theorem}[section]

\newtheorem{prop}[thm]{Proposition}
\newtheorem{hyp}[thm]{Hypothesis}

\def\1{{\rm l}\hskip -0.21truecm 1}

\begin{document}
\title{Moment bounds for some fractional stochastic heat equations on the  ball}
 \date{\today}

\author{Eulalia Nualart}

\address{Universitat Pompeu Fabra and Barcelona Graduate School of Economics, Ram\'on Trias Fargas 25-27, 08005
Barcelona, Spain.}
\email{eulalia.nualart@upf.edu}
\thanks{Research supported in part by the European Union programme FP7-PEOPLE-2012-CIG under grant agreement 333938}


\date{\today}


\begin{abstract}
In this paper, we obtain upper and lower bounds for the moments of the solution to a class of fractional stochastic heat equations on the ball driven by a Gaussian noise which is white in time, and with a spatial correlation in space of Riesz kernel type. We also consider the space-time white noise case on an interval.
\end{abstract}

\maketitle

\section{Introduction}

Consider the fractional stochastic heat equation on the unit ball $D=\{y \in \R^d: \vert y \vert < 1\}$,  $d \geq 1$, with zero Dirichlet
boundary conditions:
\begin{equation}\label{eq10}
\begin{cases}
\partial_t u_t(x)=-(-\Delta)^{\alpha/2}u_t(x)+\lambda \sigma(u_t(x)) \dot{W}(t,\,x)\qquad  x \in D, &t>0,\\
u_t(x)=0 \qquad x \in D^c, \; \;t>0, &
\end{cases}
\end{equation}
and the  initial condition is a measurable and bounded function $u_0:D \rightarrow \R_+$.
The operator $-(-\Delta)^{\alpha/2}$, where $0<\alpha \leq 2$, is the $L^2$-generator of a 
symmetric $\alpha$-stable process  killed when exiting the ball $D$. 
The coefficient $\sigma:\R\rightarrow \R$ is a globally Lipschitz function.
The Gaussian noise $\dot{W}(t,x)$ is white in time and coloured in space,
that is,  
\begin{equation} \label{f1}
\E \left( \dot{W}(t,x) \dot{W}(s,y)\right)=\delta_0(t-s) f(x-y),
\end{equation}
where $f:\R^d \rightarrow \R_+$ is a nonnegative definite (generalized) function whose Fourier transform $\hat{f}=\mu$ is a tempered measure.
Finally, the parameter $\lambda>0$ measures the level of the noise.

Following Walsh \cite{Walsh}, we define the mild solution to equation (\ref{eq10}) as the random field $u=\left\{ u_t(x)\right\}_{t> 0, x \in D}$
satisfying
 \begin{equation}\label{mild:coloured}
u_t(x)=
\int_{D} u_0(y)p_D(t,x,y)\,d y+ \lambda \int_{D}\int_0^t p_D(t-s,\,x,\,y) \sigma(u_s(y)) W(d s,d y),
\end{equation}
where $p_D(t,x,y)$ denotes the Dirichlet fractional heat kernel on $D$,
and the stochastic integral is understood
in an extended It\^o sense.  

Following Dalang \cite{Dalang}, it is well-known (see also \cite[Appendix]{LTF17} and \cite{minicourse}), that if the spectral measure $\mu$ satisfies that
\begin{equation} \label{g}
\int_{\R^d} \frac{\mu(d\xi)}{1+\vert \xi \vert^{\alpha}} < \infty,
\end{equation}
then there exists a unique 
random field solution $u$ to equation (\ref{mild:coloured}). Moreover, for all $p \geq 2$ and $T>0$, 
$$
\sup_{t \in [0,T], x \in D} \E\vert u_t(x) \vert^p <\infty.
$$

Examples of spatial correlations satisfying (\ref{g}) are:
\begin{enumerate}
\item[\textnormal{1.}] The Riesz kernel $f(x)=\vert x \vert^{-\beta}$, $0<\beta<d$. In this case, $\mu(d\xi)=c \vert \xi \vert^{-(d-\beta)}d\xi$ and it is easy to check that condition (\ref{g}) holds whenever $\beta < \alpha$.

\item[\textnormal{2.}] The fractional kernel $f(x)=\prod_{i=1}^d \vert x_i \vert^{2H_i-2}$, where $H_i \in (\frac12, 1)$ for $i=1,...,d$. In this case, $\mu(d\xi)=c \prod_{i=1}^d\vert \xi_i \vert^{1-2H_i} d \xi$ and condition (\ref{g}) holds whenever $\sum_{i=1}^d H_i>d-\frac{\alpha}{2}$.

\item[\textnormal{3.}] The Bessel kernel $f(x)=\int_0^{\infty} y^{\frac{\eta-d}{2}} e^{-y} e^{-\frac{\vert x \vert^2}{4y}} dy$. In this case, $\mu(d\xi)=c (1+\vert \xi \vert^2)^{-\eta/2} d\xi$ and condition (\ref{g}) holds whenever $\eta> d-\alpha$.

\item[\textnormal{4.}] The space-time white noise case $f=\delta_0$. In this case, $\mu(d\xi)=d\xi$, and
(\ref{g}) is only satisfied when $\alpha>d$, that is, $d=1$ and $1<\alpha \leq 2$.
\end{enumerate}

Recall that the fractional heat kernel $p_D(t,x,y)$ has spectral decomposition
$$
p_D(t,x,y)=\sum_{n=1}^{\infty} e^{-\mu_n t} \Phi_n(x) \Phi_n(y)\quad\text{for all}\quad x,\,y \in D,\quad t>0,
$$
where $\{\Phi_n\}_{n \geq 1}$ is an orthonormal basis of $L^2(D)$ and $0<\mu_1<\mu_2\leq \mu_3 \leq \cdots$ is a sequence of positive numbers, such that for every $n \geq 1$
\begin{equation}\label{eq:e}
\begin{cases}
-(-\Delta)^{\alpha/2} \Phi_n(x)=-\mu_n \Phi_n(x) \quad & x \in D,\\
\Phi_n(x)=0 \quad &x \in D^c.
\end{cases}
\end{equation}
Some properties of these families of eigenvalues and eigenfunctions are the following. From \cite[Theorem 2.3]{Blu}, there exist positive constants $c_1$ and $c_2$ such that for every $n \geq 1$,
$$
c_1 n^{\alpha/d} \leq \mu_n \leq c_2 n^{\alpha/d}.
$$
Moreover, from \cite[Theorem 4.2]{Chen2}, there exists $c>1$ such that
for all $x \in D$,
\begin{equation} \label{est}
 c^{-1} (1-\vert x\vert)^{\alpha/2}\leq \Phi_1(x) \leq c (1-\vert x\vert)^{\alpha/2},
\end{equation}
In the case that $d=1$, $\alpha=2$, and $D=(-1,1)$, we have $\Phi_n(x)= \sin(\frac{n\pi x}{2})$ and $\mu_n=(\frac{n\pi}{2})^2$. 

The aim of this paper is to obtain upper and lower bounds in terms of $t>0$ and $\lambda>0$ for the moments of the solution to equation (\ref{mild:coloured}). For this, we need some further assumptions.
We consider the following class of covariances that generalizes the Riesz kernel.
\begin{hyp} \label{h0}
There exist  positive constants $c_1,c_2$ and $0<\beta < \alpha \wedge d$ such that for all $x \in \R^d$,
$$
c_1 \vert x \vert^{-\beta} \leq f(x) \leq c_2 \vert x \vert^{-\beta}.
$$
\end{hyp}
Since we are interested in upper and lower bounds for the moments, we also need the following assumption on the coefficient $\sigma$.
\begin{hyp} \label{h1}
There exist positive constants $l_\sigma$ and $L_\sigma$ such that for all $x \in \R^d$,
\begin{equation*} 
l_\sigma|x|\leq  |\sigma(x)| \leq L_\sigma|x|.
\end{equation*} 
\end{hyp} 
Finally, in order to obtain the lower bound, we need the following assumption on the initial condition.
\begin{hyp} \label{h2}
There exists $\epsilon \in (0,\frac12)$ such that
$$
\inf_{x \in D_{\epsilon}} u_0(x) >0,
$$
where $D_{\epsilon}=\{y \in \R^d: \vert y \vert \leq 1-\epsilon\}$.
\end{hyp}
Essentially Hypothesis \ref{h2} says that there exists a closed set of positive measure inside $D$ where the initial condition stays positive.
Hypotheses \ref{h1} and \ref{h2} are usual when studying intermittency properties of SPDEs.

We are now ready to state the main result of this paper.
\begin{thm} \label{12}
Assume Hypothesis \ref{h2}. 
\begin{enumerate}
\item[\textnormal{a)}] If $f$  satisfies Hypothesis \ref{h0} and $\sigma(x)=x$, then for all $p \geq 2$, there exist positive constants $c_1$, $\overline{c}_1$, $c_2(\epsilon)$, $\overline{c}_2( \epsilon)$ such that for all $\lambda>0$,
\begin{equation*}
\overline{c}_2^p e^{p t\left(c_2\lambda^{\frac{2\alpha}{\alpha-\beta}} -\mu_1 \right)} \leq  \inf_{x \in D_{\epsilon}}\E|u_t(x)|^p\leq \sup_{x \in D}\E|u_t(x)|^p\leq \overline{c}_1^p e^{p t\left(c_1 p^{\frac{\alpha}{\alpha-\beta}} \lambda^{\frac{2\alpha}{\alpha-\beta}} -\mu_1 \right)}.
\end{equation*}

\item[\textnormal{b)}] If $f=\delta_0$ and $\sigma$ satisfies Hypothesis \ref{h1}, 
then for all $p \geq 2$ and $\delta>0$, there exist positive constants $c_1(L_{\sigma})$, $\overline{c}_1(\delta)$, $c_2(\epsilon,\ell_{\sigma})$, $\overline{c}_2( \epsilon)$ such that for all $\lambda>0$,
\begin{equation*}
\overline{c}_2^p e^{p t\left(c_2\lambda^{\frac{2\alpha}{\alpha-1}} -\mu_1 \right)} \leq  \inf_{x \in D_{\epsilon}}\E|u_t(x)|^p\leq \sup_{x \in D}\E|u_t(x)|^p\leq \overline{c}_1^p e^{p t\left(c_1 z_p^{\frac{2\alpha}{\alpha-1}} \lambda^{\frac{2\alpha}{\alpha-1}} -(1-\delta)\mu_1 \right)},
\end{equation*}
where $z_p$ is the optimal constant in Burkholder-Davis-Bundy's inequality (see \cite{FK09}).
\end{enumerate}

 Both upper bounds hold for all $t>0$ while both lower bounds holds for all $t>c(\alpha)\lambda^{-\frac{2\alpha}{\alpha-1}}.$ When $\alpha=2$, both lower bounds
hold for all $t>0$.
\end{thm}

Several remarks are in order. Observe that the bounds are not sharp in $p$ because the proof of  the lower bound is based 
on a second moment argument. However, as explained below, we are mainly interested in the  
dependence on $\lambda$, $t$ and $\mu_1$ of the moment bounds. Observe also that in the multidimensional case, we only consider the case $\sigma(x)=x$ known as parabolic Anderson model, since the method used in the space-time white noise case does not seem to apply in the multidimensional space setting. Instead, have used the Wiener-chaos expansion of the solution which is very suitable when $\sigma(x)=x$. Moreover, this allows to obtain sharper bounds in terms of $\mu_1$ in the exponential, since in this case we don't have the term $1-\delta$ in the upper bound.

The lower bound of Theorem \ref{12}b) when $\alpha=2$  was already obtained in the recent paper \cite{Xie}. Thus, Theorem \ref{12} extends this lower bound to the fractional Laplacian and higher space dimensions, and provides an upper bound of a similar type. Remark that Theorem \ref{12} can be easily extended to the ball of radius $R>0$. However, the extension to a bounded domain is not straightforward, because of the argument used in the proof of Proposition \ref{yu}.

A direct consequence of Theorem \ref{12} are the following bounds for the moment-type Lyapunov upper and lower exponents, in terms of $\lambda>0$. In case a), we obtain that for all $\lambda>0$,
\begin{equation} \label{lipu}\begin{split}
p \left(c_2\lambda^{\frac{2\alpha}{\alpha-\beta}}-\mu_1\right) &\leq  \liminf_{t\rightarrow \infty}\frac{1}{t}\log \inf_{x\in D_{\epsilon}}\E|u_t(x)|^p \\
&\leq \limsup_{t\rightarrow \infty}\frac{1}{t}\log \sup_{x \in  D}\E|u_t(x)|^p\leq p\left(\tilde{c}_1(p) \lambda^{\frac{2\alpha}{\alpha-\beta}}-
\mu_1\right),
\end{split}
\end{equation}
where $\tilde{c}_1(p)=c_1 p^{\frac{\alpha}{\alpha-\beta}}.$
Similar bounds hold for case b).
Recall from \cite{FK09} that $u$ is said to be weakly intermittent if for all $x \in D$, $$\limsup_{t\rightarrow \infty}\frac{1}{t}\log \E|u_t(x)|^2>0,$$
 and for all $p \geq 2$, and $x \in D$,
 $$\limsup_{t\rightarrow \infty}\frac{1}{t}\log \E|u_t(x)|^p<\infty.$$
 Heuristically, this phenomenon says that the solution $u$ will be concentrated  into a few very high peaks when $t$ is large. See \cite{FK09}
 and the references therein for a more detailed explanation of this phenomenon.
The bounds in (\ref{lipu}) show that in case a), if $\lambda \leq(\mu_1/\tilde{c}_1(p))^{\frac{\alpha-\beta}{2\alpha}}$, then
the solution to equation (\ref{eq10}) is not weakly intermittent, while
if  $\lambda\geq(\mu_1/c_2)^{\frac{\alpha-\beta}{2\alpha}}$, then the solution is weakly intermittent. A similar result holds for case b).

Intermittency for equations of the type (\ref{eq10}) but in all $\R^d$ have been largely studied in the literature, see e.g. \cite{FK13, HH15, BC16, FLO17}. However, much less is known in the case of bounded domains. 
In the recent paper \cite{FN14} (see also \cite{FGN16} for the extension to the fractional Laplacian), it is shown the existence of $\lambda_0(\mu_1)>0$ such that for all $\lambda<\lambda_0$,
\begin{equation*} 
-\infty<\limsup_{t\rightarrow \infty}\frac{1}{t}\log \sup_{x \in D}\E|u_t(x)|^p<0,
\end{equation*}
and  for all $\epsilon>0$, the existence of a $\lambda_1(\mu_1,\epsilon)>0$ such that for all $\lambda>\lambda_1$,
\begin{equation*} 
0<\liminf_{t\rightarrow \infty}\frac{1}{t}\log\inf_{x\in D_{\epsilon}}\E|u_t(x)|^p<\infty.
\end{equation*}
Here $u$ is the solution to equation (\ref{eq10}) with a general spatial covariance function $f$, and $\sigma$ satisfying Hypothesis \ref{h1}.
However, $\lambda_0$ and $\lambda_1$ are not explicit in those papers.
Therefore, Theorem \ref{12} provides an extension of these results with an explicit dependence of $\lambda_0$ and $\lambda_1$ in terms of $\mu_1$. Observe that the results in \cite{FN14, Xie, FGN16} already imply a dicotomy on the intermittency of the solution depending on large and small values of $\lambda$. In this paper, precise bounds for the moments for $t$ fixed are proved, which in particular imply more accurate estimates on $\lambda$ to deduce intermittency or non-intermittency of the solution. Observe also that this dicotomy phenomenon does not occur if one considers the same equation (\ref{eq10}) in $\R^d$ or in $D$ but with Newmann boundary conditions. In those cases, the solution  is  weakly intermittent for all $\lambda>0$, see e.g. \cite{FN14, FLO17}.

Theorem \ref{12}a) also implies that for all $p \geq 2$, $t>0$ and $x \in D_{\epsilon}$,
$$
\lim_{\lambda \rightarrow \infty} \frac{\log \log \E \vert u_t(x) \vert^p}{\log \lambda}=\frac{2\alpha}{\alpha-\beta},
$$
and similarly for the case b),
which is known as the excitation index of the solution introduced by \cite{KK14}. This result was obtained in \cite{LTF17} for $p=2$, $f$ the Riesz kernel and $\sigma$ satisfying Hypothesis \ref{h1}. See also \cite{KK14, FJ14} for previous results when $\alpha=2$ and $W$ is space-time white noise. The results in \cite{FJ14, LTF17} were the first that used the Gronwall's inequalities stated in Propositions \ref{esti} and \ref{estii} to show these type of results.
The proof of Theorem \ref{12}b) will be also based in those inequalities.

Consider now the deterministic heat equation $\partial_t u= \Delta u+\lambda u$ on a bounded domain $\mathcal{O}$ in $\R^d$, $d \leq 3$, with
smooth boundary and Dirichlet boundary condition $u_t(x)=0, x\in \partial \mathcal{O},  t>0$, and intial condition $u_0(x)=f(x)$, $f \in L^2 (\mathcal{O})$. It is shown in \cite{K99} that if $k_0$ is the smallest integer such that $\langle f, e_{k_0} \rangle \neq 0$, then
$$\limsup_{t \rightarrow \infty} \frac{1}{t} \log \Vert u_t \Vert_{L^2(\mathcal{O})}=\lambda-\mu_{k_0}.$$
In the same paper, the equation
$\partial_t u_t(x)= \Delta u_t(x)+\lambda u_t(x) dW_t,$ 
where $W_t$ is a real-valued Wiener process is also considered. In this case, following similar  computations as in that paper, 
it is easy to show that
$$
\limsup_{t\rightarrow \infty}\frac{1}{t}\log \sqrt{\E \Vert u_t \Vert^2_{L^2(\mathcal{O})}}=\frac{\lambda^2}{2}-\mu_{k_0}.
$$
Hence, the dycotomy phenomenon is also present in the deterministic case and the space independent white noise case. Observe that in those case we have precise
expressions for the Lyapunov exponents. For our equation (\ref{eq10}), 
even in the space-time white noise case and parabolic Anderson model, obtaining an explicit expression
for the upper second moment type Lyapunov exponent remains an open problem. Theorem \ref{12} gives a first hint of the general form of this expression.

The rest of the paper is organized as follows. Section 2 is devoted to define rigorously the Gaussian noise $W$, and the Wiener-chaos expansion
of square integrable random variables. In Section 3 we prove several heat kernel estimates that are needed for the proof of Theorem \ref{12}, and are also interesting in their own right.
Section 4 is devoted to the proof of Theorem \ref{12}. Finally, in the Appendix we recall some heat kernel estimates and fractional Gronwall's inequalities used in the paper.

\section{The Gaussian noise $W$} \label{sub1}

Let $\mathcal{D}(\R_+\times \R^d)$ be the space of real-valued infinitely differentiable functions with compact support.
Following \cite{Dalang} and \cite{minicourse}, on a complete probability space $(\Omega, \mathcal{F}, \P)$, we consider a centered Gaussian family of random variables $\{W(\varphi), \varphi \in \mathcal{D}(\R_+ \times \R^d)\}$ with covariance
$$
\E\left[ W(\varphi) W(\psi)\right]=\int_{\R_+ \times \R^{2d}} \varphi(t,x) \varphi(t,y) f(x-y) dx dy dt,
$$
where $f$ is as in (\ref{f1}).
Let $\mathcal{H}$ be the completion of $\mathcal{D}(\R_+\times \R^d)$ with respect to the inner product
$$
\langle \varphi, \psi \rangle_{\mathcal{H}}=\int_{\R_+ \times \R^{2d}} \varphi(t,x) \varphi(t,y) f(x-y) dx dy dt.
$$
The mapping $\varphi \mapsto W(\varphi)$ defined in $\mathcal{D}(\R_+ \times \R^d)$ extends to a linear isometry 
between $\mathcal{H}$ and the Gaussian space spanned by $W$. We will denote the isometry by 
$$
W(\phi)=\int_{\R_+ \times \R^d} \phi(t,x) W(dt,dx), \qquad \phi \in \mathcal{H}.
$$
Notice that if $\varphi, \psi \in \mathcal{H}$, then $\E\left[ W(\varphi) W(\psi)\right]=\langle \varphi, \psi \rangle_{\mathcal{H}}$. Moreover,
$\mathcal{H}$ contains the space of measurable functions $\phi$ on $\R_+ \times \R^d$ such that
$$
\int_{\R_+ \times \R^{2d}} \vert \phi(t,x) \phi(t,y)\vert  f(x-y) dx dy dt < \infty.
$$

When handling equation (\ref{mild:coloured}) with $\sigma(x)=x$, we will make use of its chaos expansion.
For any integer $n \geq 1$, we denote by ${\bf H}_n$ the $n$th Wiener chaos of $W$. Recall that ${\bf H}_0$ is simply $\R$ and for $n \geq 1$,
${\bf H}_n$ is the closed linear subspace of $L^2(\Omega)$ generated by the random variables
$$
\{H_n(W(h), h \in \mathcal{H}, \Vert h \Vert_{\mathcal{H}}=1\},
$$ 
where $H_n$ is the $n$th Hermite polynomial. For any $n \geq 1$, we denote by $\mathcal{H}^{\otimes n}$ (resp. $\mathcal{H}^{n}$) the
$n$th tensor product (resp. the $n$th symmetric tensor product) of $\mathcal{H}$. Then, the mapping $I_n(h^{\otimes n})=H_n(W(h))$
can be extended to a linear isometry bewteen  $\mathcal{H}^{\otimes n}$ (equipped with the modified norm $\sqrt{n!} \Vert \cdot \Vert_{\mathcal{H}^{\otimes n}})$ and ${\bf H}_n$. 

Let $\mathcal{F}^W$ the $\sigma$-field generated by $W$. Then, any $\mathcal{F}^W$-measurable random variable $F$ in $L^2(\Omega)$
can be expressed as
$$
F=\E(F)+\sum_{n=1}^{\infty} I_n(f_n),
$$
where the series converges in $L^2(\Omega)$, and the elements $f_n \in \mathcal{H}^{\otimes n}$ are determined by $F$.
This identity is called the Wiener-chaos expansion of $F$.

\section{Heat kernel estimates}

As a consequence of Theorems \ref{kernele} and \ref{kernele2} in the Appendix, the following upper and lower bounds for the fractional heat kernel on $D$ hold.
\begin{prop} \label{yu}
For any $\epsilon \in (0, \frac12)$, there exist positive constants $c_1(\epsilon)$, $c_2(\epsilon)$ and $c_3(\epsilon)$ such that for all $x \in D_{\epsilon}$ and $t>0$,
\begin{equation} \label{p1}
\int_{D_{\epsilon}} p_D(t,x,y)d y \geq c_1 e^{-\mu_1 t},
\end{equation}
for all $x \in D_{\epsilon}$ and $t>0$, 
\begin{equation} \label{p2}
\int_{D_{\epsilon}} p_D^2(t,x,y)d y \geq c_2 e^{-2\mu_1 t}t^{-d/\alpha},
\end{equation}
and if $f$ satisfies Hypothesis \textnormal{(\ref{h0})}, then for all $x,w \in D_{\epsilon}$ and $t>0$ such that $\vert x -w \vert \leq t^{\alpha}$,
\begin{equation} \label{p2bis}
\int_{D_{\epsilon} \times D_{\epsilon}} p_D(t,x,y) p_D(t,w,z) f(y-z) d y dz\geq 
c_3 e^{-2\mu_1 t} t^{-\beta/\alpha}.
\end{equation}
\end{prop}

\begin{proof}
We start assuming $\alpha=2$. From Theorem \ref{kernele} and (\ref{est}), for all $x \in D_{\epsilon}$ and $t>0$,
\begin{equation*}
\begin{split}
\int_{D_{\epsilon}} p_D(t,x,y)d y
&\geq c  \int_{D_{\epsilon}} \min\left(1, \frac{\epsilon^2}{1 \wedge t} \right) e^{-\mu_1 t} \frac{e^{-c \frac{\vert x-y \vert^2}{t}}}{1 \wedge t^{d/2}} dy \\
&\geq c e^{-\mu_1 t}\left(\int_{D_{\epsilon}}  \min\left(1, \frac{\epsilon^2}{t} \right)\frac{e^{-c \frac{\vert x-y \vert^2}{t}}}{t^{d/2}} {\bf 1}_{\{t<1\}} dy+\int_{D_{\epsilon}} e^{-c \frac{\vert x-y \vert^2}{t}} {\bf 1}_{\{t\geq 1\}}dy\right).
\end{split}
\end{equation*}
The second integral in the last display is lower bounded by $c(\epsilon)= {\bf 1}_{\{t\geq 1\}}$. The first one equals
\begin{equation*}
\int_{D_{\epsilon}}  \frac{e^{-c \frac{\vert x-y \vert^2}{t}}}{t^{d/2}} {\bf 1}_{\{t<\epsilon^2\}} dy+\epsilon^2 \int_{D_{\epsilon}}  \frac{e^{-c \frac{\vert x-y \vert^2}{t}}}{t^{1+\frac{d}{2}}} {\bf 1}_{\{\epsilon^2\leq t<1\}} dy.
\end{equation*}
The second integral in the last display is lower bounded by $c(\epsilon){\bf 1}_{\{\epsilon^2\leq t<1\}} $, while the second is lower bounded by
\begin{equation*}
\int_{D_{\epsilon}}  \frac{e^{-c \frac{\vert x-y \vert^2}{t}}}{t^{d/2}}  {\bf 1}_{\{\vert x-y \vert^2 \leq t < \epsilon^2\}} dy.
\end{equation*}
We now observe that for all $x \in D_{\epsilon}$ and $t<\epsilon^2$, since $\epsilon < \frac12$,
$$
\text{Vol} \{y \in D_{\epsilon}: \vert x-y \vert^2 \leq t \} = \text{Vol}(B_x(\sqrt{t}) \cap D_{\epsilon})\geq B_0\left(\frac{\sqrt{t}}{2}\right)=c_d t^{d/2},
$$
where $B_x(r)=\{ y \in \R^d: \vert y -x \vert \leq \sqrt{t}\}$,  for $x \in \R^d$ and $r>0$, and Vol denotes the $d$-dimensional volume.
 Therefore,
$$
\int_{D_{\epsilon}}  \frac{e^{-c \frac{\vert x-y \vert^2}{t}}}{t^{d/2}}  {\bf 1}_{\{\vert x-y \vert^2 \leq t < \epsilon^2\}} dy\geq c {\bf 1}_{\{t<\epsilon^2\}},
$$
which proves (\ref{p1}) for $\alpha=2$.

Following along the same lines, we get that 
\begin{equation*} \begin{split}
\int_{D_{\epsilon}}  p_{D}^2(t,x,y)d y
&\geq c e^{-2\mu_1 t} \left\{ \int_{D_{\epsilon}}  \frac{e^{-c \frac{\vert x-y \vert^2}{t}}}{t^d}  {\bf 1}_{\{\vert x-y\vert^2 \leq t<\epsilon^2\}}dy +  {\bf 1}_{\{t\geq\epsilon^2\}}\right\}\\
&\geq c e^{-2\mu_1 t} \left\{ t^{-d/2} {\bf 1}_{\{t<\epsilon^2\}}+ {\bf 1}_{\{t\geq\epsilon^2\}}\right\} \\
&\geq c e^{-2\mu_1 t}  t^{-d/2},
\end{split}
\end{equation*}
which shows (\ref{p2}) when $\alpha=2$.

We next show (\ref{p2bis}) for $\alpha=2$. We assume that $f$ satisfies Hypothesis \ref{h0}. From Theorem \ref{kernele} and (\ref{est}), for all $x, w \in D_{\epsilon}$ and $t>0$, 
\begin{equation*}  \begin{split}
&\int_{D_{\epsilon} \times D_{\epsilon}} p_D(t,x,y) p_D(t,w,z) f(y-z) d y dz\\
&\geq c e^{-2\mu_1 t} \left\{ \int_{D_{\epsilon}\times D_{\epsilon}}  \frac{e^{-c \frac{\vert x-y \vert^2}{t}}}{t^{d/2}}  {\bf 1}_{\{\vert x-y\vert^2 \leq t<\epsilon^2\}} \frac{e^{-c \frac{\vert w-z \vert^2}{t}}}{t^{d/2}}  {\bf 1}_{\{\vert w-z\vert^2 \leq t\}} \vert y-z \vert^{-\beta}  dy dz+  {\bf 1}_{\{t\geq\epsilon^2\}}\right\}.
\end{split}
\end{equation*}
Next observe that since $\vert x-w\vert <\sqrt{t}$, $\vert x-y\vert^2 \leq t$, and $\vert w-z\vert^2 \leq t$, we get that $\vert y-z \vert^{-\beta}\geq t^{-\beta/2}$.
Therefore, proceeding as above, we conclude that if $x,w \in D_{\epsilon}$ and $t>0$ are such that $\vert x-w\vert <\sqrt{t}$, then
\begin{equation*}  \begin{split}
\int_{D_{\epsilon} \times D_{\epsilon}} p_D(t,x,y) p_D(t,w,z) f(y-z) d y dz&\geq  c e^{-2\mu_1 t} \left\{ t^{-\beta/2} {\bf 1}_{\{t<\epsilon^2\}}+ {\bf 1}_{\{t\geq\epsilon^2\}}\right\} \\
&\geq c e^{-2\mu_1 t}  t^{-\beta/2},
\end{split}
\end{equation*}
which concludes the proof of (\ref{p2}) when $\alpha=2$.

We now assume $\alpha \in (1,2)$. Similarly as above, appealing to Theorem \ref{kernele2} and (\ref{est}), for all $x \in {D_{\epsilon}} $ and $t>0$,
\begin{equation*}
\begin{split}
&\int_{D_{\epsilon}}  p_D(t,x,y)d y \\
&\geq c e^{-\mu_1 t} \bigg\{ \int_{D_{\epsilon}}   \min\left(1, \frac{\epsilon^\alpha}{t} \right)\min \left( t^{-1/\alpha}, \frac{t}{\vert x-y \vert^{\alpha+d}} \right) {\bf 1}_{\{t<1\}}dy +{\bf 1}_{\{t \geq 1\}} \bigg\}.
\end{split}
\end{equation*}
The integral in the last display equals
\begin{equation*}
\begin{split}
 \int_{{D_{\epsilon}} }  \min \left( t^{-1/\alpha}, \frac{t}{\vert x-y \vert^{\alpha+d}} \right) {\bf 1}_{\{t<\epsilon^{\alpha}\}} dy + \epsilon^{\alpha}\int_{{D_{\epsilon}} }  \min \left( t^{-1/\alpha-1}, \frac{1}{\vert x-y \vert^{\alpha+d}} \right)  {\bf 1}_{\{\epsilon^{\alpha}\leq t <1\}} dy.
\end{split}
\end{equation*}
The second integral in the last display is lower bounded by $c(\epsilon) {\bf 1}_{\{\epsilon^{\alpha}\leq t <1\}}$, while the first one is lower bounded by 
\begin{equation*}
\begin{split}
 \int_{{D_{\epsilon}} }  t^{-d/\alpha} \min \left( t^{(d-1)/\alpha}, \left(\frac{t^{1/\alpha}}{\vert x-y \vert} \right)^{d+\alpha} \right) {\bf 1}_{\{\vert x-y \vert^{\alpha}<t<\epsilon^{\alpha}\}} dy.
\end{split}
\end{equation*}
As before, for all $x \in D_{\epsilon}$ and $t<\epsilon^{\alpha}$, since $\epsilon < \frac12$,
$$
\text{Vol} \{y \in D_{\epsilon}: \vert x-y \vert^{\alpha} \leq t \} = \text{Vol}(B_x(t^{1/\alpha}) \cap D_{\epsilon})\geq B_0\left(\frac{t^{1/\alpha}}{2}\right)=c_d t^{d/\alpha},
$$
which concludes the proof of (\ref{p1}) for $\alpha \in (1,2)$.
Similarly, 
\begin{equation*}
\begin{split}
&\int_{{D_{\epsilon}} } p_{D}^2(t,x,y)d y\\
&\geq c 
e^{-2\mu_1 t} \bigg\{  \int_{{D_{\epsilon}} } 
 \min\left(1, \frac{\epsilon^{2\alpha}}{t^2} \right)\min \left( t^{-2/\alpha}, \frac{t^2}{\vert x-y \vert^{2(\alpha+d)}} \right) {\bf 1}_{\{t<1\}} dy
+{\bf 1}_{\{t \geq 1\}} \bigg\} \\
&\geq c e^{-2\mu_1 t}\left\{  t^{-d/\alpha} {\bf 1}_{\{t<\epsilon^{\alpha}\}} 
+{\bf 1}_{\{t \geq \epsilon^{\alpha} \}} \right\} \\
&\geq c  e^{-2\mu_1 t}  t^{-d/\alpha},
\end{split}
\end{equation*}
which shows (\ref{p1}) for $\alpha \in (1,2)$.

Finally, for all $x,w \in D_{\epsilon}$ and $t>0$ such that $\vert x-w\vert <t^{1/\alpha}$,
\begin{equation*}  \begin{split}
\int_{D_{\epsilon} \times D_{\epsilon}} p_D(t,x,y) p_D(t,w,z) f(y-z) d y dz&\geq  c e^{-2\mu_1 t} \left\{ t^{-\beta/\alpha} {\bf 1}_{\{t<\epsilon^{\alpha}\}}+ {\bf 1}_{\{t\geq\epsilon^\alpha\}}\right\} \\
&\geq c e^{-2\mu_1 t}  t^{-\beta/\alpha},
\end{split}
\end{equation*}
which proves (\ref{p2bis}) when $\alpha \in (1,2)$.
\end{proof}

\begin{prop} \label{ub}
For all $\delta>0$, there exist $c_1,c_2(\delta)>0$ such that for all $x,w \in D$ and $t>0$,
\begin{equation} \label{p11}
\int_{D}  p_D(t,x,y)dy\leq c_1 e^{-\mu_1 t},
\end{equation}
and
\begin{equation} \label{p22}
\int_{D \times D}  p_D(t,x,y)p_D(t,w,z) f(y-z)d y dz\leq c_2 e^{-(2-\delta)\mu_1 t} t^{-a/\alpha},
\end{equation}
where \begin{equation*}
a=\begin{cases}
d, & \text{if} \quad f=\delta_0,\\
\beta, & \text{if} \quad f \text{ satisfies Hypothesis \textnormal{\ref{h0}}}.
\end{cases}
\end{equation*}
\end{prop}

\begin{proof}
We first assume $\alpha=2$. By Theorem \ref{kernele}, for all $x \in D$ and $t>0$,
\begin{equation*}
\begin{split}
\int_{D} p_D(t,x,y)d y
&\leq c  e^{-\mu_1 t}\int_{D}  \frac{e^{-c \frac{\vert x-y \vert^2}{t}}}{1 \wedge t^{d/2}} dy \leq c e^{-\mu_1 t},
\end{split}
\end{equation*}
which shows (\ref{p11}). Let $f=\delta_0$.
By the semigroup property and Theorem \ref{kernele}, for all $\delta>0$,
\begin{equation*}
\int_D p_{D}^2(t,x,y)dy
= p_{D}(2t,x,x) \leq c  e^{-2\mu_1 t}\frac{1}{1 \wedge t^{d/2}} \leq c(\delta) e^{-(2-\delta)\mu_1 t} t^{-d/2}.
\end{equation*}
Finally, by Theorem \ref{kernele}, when $f$ satisfies Hypothesis \textnormal{\ref{h0}}, we get that
\begin{equation*} \begin{split} 
&\int_{D \times D}  p_D(t,x,y)p_D(t,w,z) f(y-z)d y dz\leq c  e^{-2\mu_1 t}\int_{D\times D}  \frac{e^{-c \frac{\vert x-y \vert^2}{t}}}{1 \wedge t^{d/2}}  \frac{e^{-c \frac{\vert w-z \vert^2}{t}}}{1 \wedge t^{d/2}}  \vert y-z \vert^{-\beta} dy dz\\
&\qquad \qquad \qquad\leq c(\delta)  e^{-(2-\delta)\mu_1 t}\int_{\R^d\times \R^d}  \frac{e^{-c \frac{\vert x-y \vert^2}{t}}}{t^{d/2}}  \frac{e^{-c \frac{\vert w-z \vert^2}{t}}}{ t^{d/2}}  \vert y-z \vert^{-\beta} dy dz\\
&\qquad \qquad \qquad\leq c(\delta)  e^{-(2-\delta)\mu_1 t} t^{-a/\alpha},
\end{split}
\end{equation*}
where we have used \cite[Lemma 4.1]{LTF17}  in the last inequality.
\end{proof}

\section{Proof of Theorem \ref{12}}

\subsection{Proof of the lower bound of Theorem \ref{12}b)}
By Jensen's inequality, for any $p\geq 2$,
\begin{equation} \label{ji}
\E|u_t(x)|^p \geq \left(\E|u_t(x)|^2\right)^{p/2}.
\end{equation}
Therefore, it suffices to prove the lower bound for $p=2$. 
Taking the second moment to the mild formulation (\ref{mild:coloured}) we obtain that for all $x \in D$ and $t>0$,
\begin{equation*}
\E|u_t(x)|^2=\left(\int_{D} u_0(y)p_D(t,x,y)d y\right)^2+\lambda^2\int_0^t\int_{D} p_{D}^2(t-s,x,y)\E|\sigma(u_s(y))|^2d y d s.
\end{equation*}
Hypothesis \ref{h2} and the heat kernel estimate (\ref{p1}) yield to
\begin{equation*}
\begin{split}
\int_D u_0(y)p_D(t,x,y)d y\geq  \inf_{y \in D_{\epsilon}} u_0(y) \int_{D_{\epsilon}} p_D(t,x,y)d y \geq c e^{-\mu_1 t}.
\end{split}
\end{equation*}
On the other hand, from Hypothesis \ref{h1}and the heat kernel estimate (\ref{p2}), we get that 
\begin{equation*}
\begin{split}
\int_0^t\int_D p_{D}^2(t-s,x,y)\E|\sigma(u_s(y))|^2d y d s &\geq  \ell^2_{\sigma}\int_0^t\int_{D_{\epsilon}} p_{D}^2(t-s,x,y)\E|u_s(y)|^2 d y d s\\
&\geq \ell^2_{\sigma}\int_0^t h_{\epsilon}(s) \int_{D_{\epsilon}} p_{D}^2(t-s,x,y) dy d s \\
&\geq c\int_0^t h_{\epsilon}(s)   e^{-2\mu_1 (t-s) } (t-s)^{-1/\alpha} d s,
\end{split}
\end{equation*}
where $
h_{\epsilon}(s):=\inf_{y \in D_{\epsilon}} \E|u_s(y)|^2.
$
Now, set $g_{\epsilon}(t)=e^{2\mu_1 t} h_{\epsilon}(t) $. The estimates above show that for all $t>0$,
\begin{equation*}
g_{\epsilon}(t) \geq c\left(1+\lambda^2 \int_0^t (t-s)^{-1/\alpha} g_{\epsilon}(s) d s\right).
\end{equation*}
Finally, Proposition \ref{esti} with $\rho=1-\frac{1}{\alpha}$ and Proposition \ref{estii} conclude the desired lower bound.

\subsection{Proof of the upper bound of Theorem \ref{12}b)}
Taking the $p$th moment to the mild formulation (\ref{mild:coloured}) and appealing to Burkholder-Davis-Bundy's and Minkowski's inequalities, it holds that for all $x \in D$ and $t>0$,
\begin{equation*} \begin{split}
\E|u_t(x)|^p&\leq 2^{p-1} \bigg\{\left(\int_D u_0(y)p_D(t,x,y)d y\right)^p \\
& \qquad \qquad +\lambda^p z^p_p\left(\int_0^t\int_D p_{D}^2(t-s,x,y) (\E \vert \sigma(u_s(y)) \vert^p)^{2/p} d y d s\right)^{p/2} \bigg\},
\end{split}
\end{equation*}
where $z_p$ is as in the statement of Theorem \ref{12}b).
Since $u_0$ is bounded, and using the heat kernel estimate (\ref{p11}), we get
\begin{equation*}
\begin{split}
\int_D u_0(y)p_D(t,x,y)d y\leq c_1  e^{-\mu_1 t}.
\end{split}
\end{equation*}
Using Hypothesis \ref{h1} and the heat kernel estimate (\ref{p22}), we obtain
\begin{equation*}
\begin{split}
\int_0^t\int_D p_{D}^2(t-s,x,y)(\E|\sigma(u_s(y))|^p)^{2/p} d y d s&\leq L^2_{\sigma}\int_0^t\int_D p_{D}^2(t-s,x,y)(\E|u_s(y)|^p)^{2/p}d y d s\\
&\leq L^2_{\sigma}\int_0^t h(s) \left(\int_D p_{D}^2(t-s,x,y)d y \right)ds \\
&\leq c \int_0^t h(s)  e^{-(2-\delta)\mu_1 (t-s) } (t-s)^{-1/\alpha}  ds,
\end{split}
\end{equation*}
where 
$
h(s)=\sup_{y \in D} (\E|u_s(y)|^p)^{2/p}.
$
The estimates above show that for all $t>0$,
$$
g(t)\leq c\left(1 + \lambda^2 z_p^2 \int_0^t  \frac{g(s)}{(t-s)^{1/\alpha}} d s\right),
$$
where $g(t)=e^{(2-\delta)\mu_1 t} h(t)$.
Finally, Proposition \ref{esti} with $\rho=1-\frac{1}{\alpha}$ concludes.

\subsection{Proof of Theorem \ref{12}a)}

In this case, following \cite{BC16}, the solution to (\ref{eq10}) has the following Wiener-chaos expansion in $L^2(\Omega)$
\begin{equation} \label{expansion}
u_t(x)= h_0(t,x)+ \sum_{n \geq 1}  \lambda^n I_n(h_n(\cdot,t,x)),
\end{equation}
where $h_0(t,x)=\int_D u_0(y)p_D(t,x,y)d y$, and for $n \geq 1$,
$I_n$ denotes the multiple Wiener integral with respect to $W$ in $\R_+^n \times D^n$, and for any $(t_1,...,t_n) \in \R_+^n$ and $x_1,...,x_n \in D$,
\begin{equation*} \begin{split}
&h_n(t_1,x_1,...,t_n,x_n, t,x)=p_D(t-t_{n},x,x_n)p_D(t_n-t_{n-1},x_{n},x_{n-1}) \\ 
&\qquad \qquad \qquad \cdots p_D(t_2-t_1,x_2,x_1)  h_0(t_1,x_1) {\bf 1}_{\{0<t_1<\cdots <t_n<t\}}.
\end{split}
\end{equation*}

Therefore,
\begin{equation*}
\E \vert u_t(x) \vert^2= \vert h_0(t,x) \vert^2+ \sum_{n \geq 1} \lambda^{2n} n! \Vert \tilde{h}_n(\cdot, t,x)\Vert^2_{\mathcal{H}^{\otimes 2}},
\end{equation*}
where $\tilde{h}_n$ denotes the symmetrization of $h_n$. That is,
\begin{equation*}
\begin{split}
& n!\Vert \tilde{h}_n(\cdot, t,x)\Vert^2_{\mathcal{H}^{\otimes 2}}=
 \int_{0<t_1<\cdots <t_n<t} \int_{D^{2n}} p_D(t-t_n,x,x_n) p_D(t-t_n,x,y_n) f(x_n-y_n) \\
&  \times p_D(t_n-t_{n-1},x_n,x_{n-1}) p_D(t_n-t_{n-1},x_{n},y_{n-1}) f(x_{n-1}-y_{n-1})\cdots p_D(t_2-t_1,x_{2},x_{1})  \\
&\times p_D(t_2-t_1,x_{2},y_{1}) f(x_{1}-y_{1})\vert h_0(t_1,x_1) \vert^2 dx_1\cdots dx_n dy_1 \cdots dy_n dt_1 \cdots dt_n.
\end{split}
\end{equation*}

Now, appealing to Propositions \ref{yu} and \ref{ub}, we get
\begin{equation*}
\begin{split}
&c_2  e^{-2\mu_1 t} 
 \int_{0<t_1<\cdots <t_n<t} (t-t_n)^{-a/\alpha} \prod_{2=1}^n (t_i-t_{i-1})^{-\beta/\alpha}dt_1 \cdots dt_n \\
&\leq n!\Vert \tilde{h}_n(\cdot, t,x)\Vert^2_{\mathcal{H}^{\otimes 2}}\leq c_1 e^{-2\mu_1 t} 
 \int_{0<t_1<\cdots <t_n<t} (t-t_n)^{-a/\alpha} \prod_{2=1}^n (t_i-t_{i-1})^{-\beta/\alpha}dt_1 \cdots dt_n.
 \end{split}
\end{equation*}
Following similar computations as in \cite{BC16, B17}, using \cite[Lemma 4.5]{HH15}, it is easy to see that the last display implies that
\begin{equation*} 
 c_2  e^{-2\mu_1 t} \sum_{n \geq 1} \lambda^{2n} 
C_2^n (n!)^{\frac{\beta}{\alpha}-1} t^{-\frac{n \beta}{\alpha}+n}  \leq \E \vert u_t(x) \vert^2\leq   
c_1  e^{-2\mu_1 t} \sum_{n \geq 0} \lambda^{2n} 
C_1^n (n!)^{\frac{\beta}{\alpha}-1} t^{-\frac{n \beta}{\alpha}+n},
\end{equation*}
and this gives the statement of Theorem \ref{12} for $p=2$. 
The lower bound for $p \geq 2$ follows using Jensen's inequality as in (\ref{ji}).
For the upper bound, as in \cite{BC16,B17}, we have that by Minkowski's inequality and the equivalence of norms in a fixed Wiener chaos, for all $p \geq 2$,
\begin{equation*}
\left(\E \vert u_t(x) \vert^p\right)^{1/p} \leq \sum_{n \geq 0} (p-1)^{n/2}
\lambda^{n} \left(n! \Vert \tilde{h}_n(\cdot, t,x)\Vert^2_{\mathcal{H}^{\otimes 2}} \right)^{1/2},
\end{equation*}
which implies the desired upper bound.

\section{Appendix}

 We make use of the following fractional Gronwall's inequalities.
\begin{prop} \label{esti}\textnormal{\cite[Lemma 7.1.1]{DH81}, \cite{FLO17}}
Let $\rho>0$ and  suppose that $g(t)$ is a locally integrable function  satisfying
\begin{equation} \label{k1}
g(t)\leq c_1+k \int_0^t (t-s)^{\rho-1} g(s) d s\quad \text{for\,all} \quad  t >0,
\end{equation}
for some positive constants $c_1,k$. Then there exist positive constants $c_2,c_3$ such that
\begin{equation*}
g(t)\leq c_2 e^{c_3  k^{1/\rho} t}\quad \text{for\,all}\quad t>0.
\end{equation*}
If instead of \textnormal{(\ref{k1})} the function is non-negative and satisfies
\begin{equation*} 
g(t)\geq c_1+k \int_0^t (t-s)^{\rho-1} g(s) d s\quad \text{for\,all} \quad  t >0,
\end{equation*}
then
\begin{equation*}
g(t)\geq c_2 e^{c_3  k^{1/\rho} t}\quad \text{for\,all}\quad t>\frac{e}{\rho} (\Gamma(\rho) k)^{-1/\rho}.
\end{equation*}
\end{prop}

The next result shows that when $\rho=\frac12$, the lower bound can be obtained for all $t>0$.
\begin{prop} \label{estii}\textnormal{\cite{FJ14}}
Let $g(t)$ be a non-negative locally integrable function  satisfying
\begin{equation*} 
g(t)\geq c_1+k \int_0^t \frac{g(s)}{\sqrt{t-s}} d s\quad \text{for\,all} \quad  t >0,
\end{equation*}
for some positive constants $c_1,k$. Then there exist positive constants $c_2,c_3$ such that
\begin{equation*}
g(t)\geq c_2e^{c_3k^2t}\quad \text{for\,all}\quad t>0.
\end{equation*}
\end{prop}

In order to achieve our goal, we use the following estimates of the Dirichlet fractional heat kernel. 
\begin{thm} \label{kernele}\textnormal{\cite[Theorem 2.2]{R13}}
Assume $\alpha=2$. There exist positive constants $C, c_1$ and $c_2$ such that for all $x, y \in D$ and $t>0$,
\begin{equation*}
\begin{split}
&C^{-1} \min\left(1, \frac{\Phi_1(x) \Phi_1(y)}{1 \wedge t} \right) e^{-\mu_1 t} \frac{e^{-c_2 \frac{\vert x-y \vert^2}{t}}}{1 \wedge t^{d/2}}\leq p_D(t,x,y)\\
&\qquad \qquad \qquad \qquad \leq C \min\left(1, \frac{\Phi_1(x) \Phi_1(y)}{1 \wedge t} \right) e^{-\mu_1 t} \frac{e^{-c_1 \frac{\vert x-y \vert^2}{t}}}{1 \wedge t^{d/2}}.
\end{split}
\end{equation*}
\end{thm}

\begin{thm} \label{kernele2} \textnormal{\cite[Theorem 1.1]{Chen}}
Assume $\alpha \in (1,2)$.
There exist a positive constant $C$ such that for all $x, y \in D$ and $t>0$,
\begin{equation*}
\begin{split}
 &C^{-1} e^{-\mu_1 t}  \big\{\min \left( 1, \frac{\Phi_1(x)}{\sqrt{t}} \right)
\min \left(1, \frac{\Phi_1(y)}{\sqrt{t}} \right) 
\min \left( t^{-d/\alpha}, \frac{t}{\vert x-y \vert^{\alpha+d}} \right) {\bf 1}_{\{t<1\}} \\ 
&\qquad \qquad \qquad +  \Phi_1(x) \Phi_1(y) {\bf 1}_{\{t \geq 1\}} \big\}\\
&\leq p_D(t,x,y)\\
&\leq Ce^{-\mu_1 t} \big\{ \min \left(1, \frac{\Phi_1(x)}{\sqrt{t}} \right)
\min \left(1, \frac{\Phi_1(y)}{\sqrt{t}} \right) 
\min \left( t^{-d/\alpha}, \frac{t}{\vert x-y \vert^{\alpha+d}} \right) {\bf 1}_{\{t<1\}} \\ 
&\qquad \qquad \qquad + \Phi_1(x) \Phi_1(y) {\bf 1}_{\{t \geq 1\}}\big\}.
\end{split}
\end{equation*}
\end{thm}


\end{document}